\documentclass[10pt]{article}%

\usepackage[utf8]{inputenc}

\usepackage{amsmath}

\usepackage{amsfonts}
\usepackage{mathrsfs}
\usepackage{amssymb, color}
\usepackage[linkcolor=black,anchorcolor=black,citecolor=black]{hyperref}
\usepackage{graphicx}
\numberwithin{equation}{section}
\usepackage[body={15.5cm,21cm}, top=3cm]{geometry}%
\providecommand{\U}[1]{\protect\rule{.1in}{.1in}}
\providecommand{\U}[1]{\protect \rule{.1in}{.1in}}
\newtheorem{theorem}{Theorem}[section]

\newtheorem{corollary}[theorem]{Corollary}

\newtheorem{definition}[theorem]{Definition}
\newtheorem{example}[theorem]{Example}

\newtheorem{lemma}[theorem]{Lemma}

\newtheorem{proposition}[theorem]{Proposition}
\newtheorem{remark}[theorem]{Remark}

\newtheorem{assumption}[theorem]{Assumption}
\newenvironment{proof}[1][Proof]{\noindent \textbf{#1.} }{\  \rule{0.5em}{0.5em}}

\def \E{\mathsf{E}}

\def \P{\mathsf{P}}

\def \Q{\mathsf{Q}}

\usepackage [numbers,sort&compress] {natbib} 

\def\d{\mathrm{d}}
\def \E{\mathsf{E}}
\def \P{\mathsf{P}}

\begin{document}
	\title{Backward Stochastic Differential Equations with Nonlinear Expectation Reflection}
	\author{ 	Hanwu Li\thanks{Research Center for Mathematics and Interdisciplinary Sciences, Shandong University, Qingdao 266237, Shandong, China. lihanwu@sdu.edu.cn.}
	\thanks{Frontiers Science Center for Nonlinear Expectations (Ministry of Education), Shandong University, Qingdao 266237, Shandong, China.}}
	\date{}
	\maketitle
	
	\begin{abstract}
        In this paper, we study a kind of constrained backward stochastic differential equations (BSDEs) such that the nonlinear expectation of the composition of a loss function and the solution remains above zero. The existence and uniqueness result is established with the help of the Skorokhod problem and the method of contraction mapping. We provide the  comparison properties for the pointwise value of the solutions and the expectation of the solutions, respectively. In addition, a similar BSDE with risk measure reflection is proposed, which can be applied to the superhedging for contingent claims under risk management constraints.
	\end{abstract}


    \textbf{MSC-classification}: 60H10
	
\section{Introduction}

El Karoui et al. \cite{EKPPQ} first introduced the notion of reflected backward stochastic differential equations (reflected BSDEs) whose dynamics evolve according to the following equation
\begin{align*}
    Y_t=\xi+\int_t^T f(s,Y_s,Z_s)ds-\int_t^T Z_s dB_s+(K_T-K_t),
\end{align*}
where $K$ is a nondecreasing process aiming to push the solution upwards such that for a given lower obstacle $L$, we have $Y_t\geq L_t$ for any $t\in[0,T]$. 
It is natural to find the minimal solution satisfying the above constraint, which can be determined by the following Skorokhod condition
\begin{align*}
    \int_0^T(Y_t-L_t)dK_t=0.
\end{align*}
Due to the importance in theoretical analysis and in applications, the reflected BSDEs have attracted a great deal of attention. To name a few, Cvitani\'{c} and Karatzas \cite{CK} investigated doubly reflected BSDEs, where the solution is forced to stay between a lower obstacle and an upper obstacle. The doubly reflected BSDEs are closely related to the Dynkin games. Peng and Xu \cite{PX} considered a more general case in which constraints may also depend on the second component of the solution, which can be used to solve the problem of superhedging with restricted admissible portfolios. To study the optimal switching problems, Hu and Tang \cite{HT} investigated multidimensional BSDEs with oblique reflections. When the nondecreasing process behaves in a nonlinear way, we may refer to Qian and Xu \cite{QX} for the reflected BSDEs with resistence, which can be applied to the superhedging problem with wealth constraint. 


It should be pointed out that the constraints in all the papers mentioned above are made pointwisely for the solution. Recently, Briand, Elie and Hu \cite{BEH} proposed a new kind of BSDE with weak constraints which takes the following form 
\begin{align*}
\E[l(t,Y_t)]\geq 0, \ t\in[0,T].
\end{align*}
In order to obtain the existence and uniqueness of solutions to the so-called mean reflected BSDEs, the process $K$ is restricted to be a deterministic function satisfying the Skorokhod condition. One of the application of the mean reflected BSDEs is the superhedging of claims under running risk management constraint, where the risk measure $\{\rho(t,\cdot)\}_{t\in[0,T]}$ is Lipschitz continuous with respect to the $L^1$-norm. 
Since then, many researchers have put considerable effort to the study of the mean reflected BSDEs. Briand and Hibon \cite{BH} established the propagation of chaos for mean reflected BSDEs using the interacting particle systems. Falkowski and Slomi\'{n}ski \cite{FS} considered the mean reflected BSDEs with two constraints, which is generalized by Li \cite{Li2024} to the case of different loss functions. Recently, Qu and Wang \cite{QW} investigated the multi-dimensional case with possibly non-convex reflection domains along inward normal direction. The readers may refer to \cite{CHM,HL,HHLLW,HMW,LN,LW} and the references therein for more aspects related with mean reflected BSDEs.

It should be emphasized that  that the BSDE with risk measure reflection in \cite{BEH} needs the Lipschitz assumption for the risk measure. However, this assumption may not always holds true (see Remark \ref{remark4.5} below). In order to overcome this problem, in the present paper, we propose the constraint by means of an appropriate nonlinear expectation $\mathcal{E}$. More precisely, we consider the following type of BSDE with nonlinear expectation reflection
\begin{equation*}
\begin{cases}
Y_t=\xi+\int_t^T f(s,Y_s,Z_s)ds-\int_t^T Z_s dB_s+K_T-K_t, \\
\mathcal{E}[l(t,Y_t)]\geq 0, \\
K\in I[0,T] \textrm{ and }\int_0^T \mathcal{E}[l(t,Y_t)]dK_t=0.
\end{cases}
\end{equation*}
As usual, the Skorokhod condition ensures the minimality of the solution, which is shown in Proposition \ref{prop3.9} for some specific driver $f$. Here, the nonlinear expectation is assumed to be dominated by a $g$-expectation with generator $g(t,z)=\kappa(|y|+|z|)$, denoted by $G^\kappa_{0,T}$ (see \cite{CP,P97} for more details about $g$-expectation). It should be pointed out that the building block for the construction of the classical mean reflected BSDEs is a nonlinear operator $L_t(X)$, which may be interpreted as the minimal strength with which the random variable $X$ can be pushed upward in order to fulfill the constraint at time $t$. Furthermore, in most of the above mentioned papers concerning mean reflected BSDEs, this operator is required to satisfy the following Lipschitz assumption, i.e., there exists a positive constant $C$, such that for any $t\in[t,T]$, we have
\begin{align*}
    |L_t(X)-L_t(Y)|\leq C\E[|X-Y|].
\end{align*}
In the present setting, although the operator $L_t(X)$ associated with the nonlinear expectation reflection may not satisfy the Lipschitz assumption, it inherits the domination property of the nonlinear expectation (see Proposition \ref{lip}). Fortunately, by the a priori estimates for BSDEs and the representation of the third component $K$, we may construct the contraction mapping and then obtain the existence and uniqueness result. Another interesting problem is how the terminal value $\xi$, the loss function $l$ and the nonlinear expectation $\mathcal{E}$ affect the solution. It is worth mentioning that since the constraint is given in nonlinear expectation instead of pointwisely, the comparison theorem may not hold in general.  We answer this question by establishing the comparison properties for some typical cases. By the motivation of the present paper, we introduce the BSDE with risk measure reflection, where the risk measure may not satisfy the Lipschitz assumption as required in \cite{BEH}. The solution to this kind of reflected BSDE coincides with the superhedging price for a contingent claim under running risk management constraint.

This paper is organized as follows. In Section 2, we first formulate the BSDEs with nonlinear expectation reflection in details. Then, we establish the existence and uniqueness result and some properties of BSDEs with nonlinear expectation reflection including the comparison theorem and the representation property. In Section 4, we introduce the BSDEs with risk measure reflection and their financial applications.

\section{$g$-expectation and related properties}

Let $T>0$ be a finite time horizon. Consider a filtered probability space $(\Omega,\mathcal{F},\mathbb{F},\P)$ satisfying the usual conditions of right continuity and completeness in which $B$ is a standard Brownian motion. For simplicity, we only consider the case of one-dimensional Brownian motion and the results still holds for the multi-dimensional case. The following notations are frequently used in this paper.

\begin{itemize}
\item $L^2(\mathcal{F}_t)$: the set of real-valued $\mathcal{F}_t$-measurable random variable $\xi$ such that $\E[|\xi|^2]<\infty$, $t\in[0,T]$.
\item $\mathcal{S}^2_{[s,t]}$: the set of real-valued adapted continuous processes $Y$ on $[s,t]$ such that $$\E\left[\sup_{r\in[s,t]}|Y_r|^2\right]<\infty.$$
\item $\mathcal{H}^2_{[s,t]}$: the set of real-valued predictable processes $Z$ such that $\E\left[\int_s^t|Z_r|^2dr\right]<\infty$.
\item $C[s,t]$: the set of continuous functions from $[s,t]$ to $\mathbb{R}$.
\item $I[s,t]$: the set of functions in $C[s,t]$ starting from the origin which is nondecreasing.
\end{itemize}
If $[s,t]=[0,T]$, $\mathcal{S}^2_{[s,t]}$ and $\mathcal{H}^2_{[s,t]}$ are denoted by $\mathcal{S}^2$ and $\mathcal{H}^2$, respectively.

For each given $X \in L^2\left(\mathcal{F}_T\right)$, consider the following backward stochastic differential equation (BSDE):
\begin{equation}\label{BSDE}
y^X(t)  =X+\int_t^T g\left(s, y^X_s, z^X_s\right) ds-\int_t^T z^X_s d B_s, 
\end{equation}
 where $g:\Omega\times[0,T]\times \mathbb{R}\rightarrow\mathbb{R}$ satisfies the following conditions:
   \begin{itemize}
       \item[(A1)] For each fixed $y,z$, $g(\cdot,\cdot,y,z)$ is progressively measurable.
       \item[(A2)]  There exists $\kappa>0$ such that for any $t\in[0,T]$ and any $y,y'\in\mathbb{R}$, $z,z'\in\mathbb{R}$
\begin{align*}
|g(t,y,z)-g(t,y,z')|\leq \kappa( |y-y'|+|z-z'|).
\end{align*}
\item[(A3)] For any $t\in[0,T]$, $g(t,0,0)=0$.
   \end{itemize}
   By the result in \cite{PP90}, the above BSDE admits a unique pair of solution $\left(y^X, z^X\right) \in \mathcal{S}^2\times \mathcal{H}^2$. 

\begin{definition}[g-expectation] The g-expectation $G_{0,T}(\cdot): L^2(\mathcal{F}_T) \longmapsto \mathbb{R}$ is defined by (see Definition 1.1 in \cite{CP}) 
$$
G_{0,T}(X):=y^X_0.
$$
Suppose that $g=\kappa(|y|+|z|)$ (resp., $g=-\kappa(|y|+|z|)$),  we write $G^{\kappa}_{0,T}(\cdot)$ (resp., $G^{-\kappa}_{0,T}(\cdot)$) instead of $G_{0,T}(\cdot)$.
\end{definition}

Under Assumptions (A1)-(A3), the $g$-expectation of the random variable $X$ also depends on the terminal time $T$ and the constant preserving property is not valid. Fortunately, some nice properties such as monotonicity, continuity still hold.

\begin{proposition}[\cite{CP}]\label{g-exp}
    Under Assumptions (A1)-(A3), the $g$-expectation has the following properties
    \begin{itemize}
        \item[(1)] If $X\leq Y$, then $G_{0,T}[X]\leq G_{0,T}[Y]$. Furthermore, if $\P(X<Y)>0$, then $G_{0,T}[X]<G_{0,T}[Y]$.
        \item[(2)] There is a constant $C_{\kappa,T}$ depending only on $\kappa, K$, such that $|G_{0,T}[X]|^2\leq C_{\kappa,T}\E[|X|^2]$.
        \item[(3)] $G_{0,T}^{-\kappa}[X-Y]\leq G_{0,T}[X]-G_{0,T}[Y]\leq G_{0,T}^{\kappa}[X-Y]$.
        \item[(4)] $G_{0,T}^\kappa[X]=-G_{0,T}^{-\kappa}[-X]$.
         \item[(5)] Let $C>0$ be a constant. We have $G_{0,T}^\kappa[CX]=CG^{\kappa}_{0,T}[X]$ and $G^{-\kappa}_{0,T}[CX]=CG_{0,T}^{-\kappa}[X]$.
    \end{itemize}
\end{proposition}

\begin{remark}\label{g-exp'}
(i) In fact, the notion of $g$-expectation was initially introduced in \cite{P97} under (A1), (A2) and a slightly stronger condition  given as follows
\begin{itemize}
\item[(A3')] For any $t\in[0,T]$ and $y\in\mathbb{R}$, $g(t,y,0)=0$.
\end{itemize}
In this case, the $g$-expectation does not depend on the terminal time and it satisfies the constant preserving property.  If $g=\kappa|z|$ (resp., $g=-\kappa|z|$),  we write $\mathcal{E}^{\kappa}(\cdot)$ (resp., $\mathcal{E}^{-\kappa}(\cdot)$) instead of $G_{0,T}(\cdot)$.

\noindent (ii) Let $C>0$ be a constant, consider the BSDE \eqref{BSDE} with terminal value $C$ and driver $g=-\kappa(|y|+|z|)$. It is easy to check that $(C e^{-\kappa(T-t)},0)_{t\in[0,T]}$ is its solution.  That is to say, although the $g$-expectation does not satisfy the constant preserving property, we may calculate that for any constant $C>0$,
    \begin{align*}
        G_{0,T}^{-\kappa}[C]=C e^{-\kappa T}.
    \end{align*}
    A direct consequence of Proposition \ref{g-exp} is that 
    \begin{align*}
        G_{0,T}^{\kappa}[-C]=-C e^{-\kappa T}.
    \end{align*}
\end{remark}

\section{BSDE with nonlinear expectation reflection}
\subsection{Problem formulation}

The objective of this paper is to study the BSDE with nonlinear expectation reflection whose parameters consist of the terminal value $\xi$, the driver (or coefficient) $f$, the loss functions $l$ and the nonlinear expectation $\mathcal{E}$, which is given as follows
\begin{equation}\label{nonlinearyz}
\begin{cases}
Y_t=\xi+\int_t^T f(s,Y_s,Z_s)ds-\int_t^T Z_s dB_s+K_T-K_t, \\
\mathcal{E}[l(t,Y_t)]\geq 0, \\
K\in I[0,T] \textrm{ and }\int_0^T \mathcal{E}[l(t,Y_t)]dK_t=0.
\end{cases}
\end{equation}
We propose the following assumptions on the parameters $f,l,\mathcal{E}$.
 
\begin{assumption}\label{assf}
The driver $f$ is a map from $\Omega\times[0,T]\times \mathbb{R}\times\mathbb{R}$ to $\mathbb{R}$. For 
each fixed $(y,z)$, $f(\cdot,\cdot,y,z)$ is progressively measurable. There exists $\lambda>0$ such that for any $t\in[0,T]$ and any $y,y',z,z'\in\mathbb{R}$
\begin{align*}
|f(t,y,z)-f(t,y',z')|\leq \lambda(|y-y'|+|z-z'|)
\end{align*}
and 
\begin{align*}
\E\left[\int_0^T |f(t,0,0)|^2dt\right]<\infty.
\end{align*}
\end{assumption}

\begin{assumption}\label{ass2}
The function $l:\Omega\times [0,T]\times\mathbb{R}\rightarrow \mathbb{R}$ is a measurable map with respect to $\mathcal{F}_T\times \mathcal{B}([0,T])\times \mathcal{\mathbb{R}}$ satisfying the following conditions.
\begin{itemize}
\item[(1)] For any fixed $(\omega,x)\in \Omega\times\mathbb{R}$, $l(\omega,\cdot,x)$ is continuous.
\item[(2)] $\E\left[\sup_{t\in[0,T]}|l(t,0)|^2\right]<\infty$.
\item[(3)] For any fixed $(\omega,t)\in \Omega\times [0,T]$, $l(\omega,t,\cdot)$ is strictly increasing and there exists two constants $0<c<C$ such that for any $x,y\in \mathbb{R}$,
\begin{align*}
c_l|x-y|\leq |l(\omega,t,x)-l(\omega,t,y)|\leq C_l|x-y|.
\end{align*}
\end{itemize}
\end{assumption}

\begin{assumption}\label{assE}
    The nonlinear expectation $\mathcal{E}:L^2(\mathcal{F}_T)\rightarrow \mathbb{R}$ satisfies the following conditions.
    \begin{itemize}
        \item[(1)] For $X\leq Y$, we have $\mathcal{E}[X]\leq \mathcal{E}[Y]$. 
        \item[(2)] There exists a constant $M>0$, such that for any $X_1,X_2\in L^2(\mathcal{F}_T)$,
        \begin{align*}
            \mathcal{E}[X_1]-\mathcal{E}[X_2]\leq MG^{\kappa}_{0,T}[X_1-X_2].
        \end{align*}
    \end{itemize}
\end{assumption}

\begin{example}\label{egofnonlinear expectation}
 (i) Clearly, the classical expectation $\E[\cdot]$ and the $g$-expectation $G_{0,T}(\cdot)$ satisfies Assumption \ref{assE}. In Remark \ref{remark4.5} below, we introduce some risk measures satisfying Assumption \ref{assE} (2).
 
\noindent (ii) Suppose $\mathcal{P}$ is a collection of probability measures $\P^{\theta}$ with Girsanov kernel $\theta\in \Theta$, where
    \begin{align*}
        \frac{d\P^{\theta}}{d\P}=\exp\left(\int_0^T \theta_s dB_s-\frac{1}{2}\int_0^T \theta_s^2 ds\right).
    \end{align*}
    Let $\Theta_\kappa$ be the collection of all progressively measurable processes $\theta$ such that $|\theta|\leq \kappa$. Assume that $\Theta\subset \Theta_\kappa$. Consider the  $\alpha$-maxmin expectation $\mathcal{E}_{\alpha,\mathcal{P}}[\cdot]$ with $\alpha\in[0,1]$ (see \cite{BLR,GMM}), i.e.,
    \begin{align*}
        \mathcal{E}_{\alpha,\mathcal{P}}[\xi]:=\alpha \sup_{\theta\in \Theta}\E^{\P^\theta}[\xi]+(1-\alpha) \inf_{\theta\in \Theta}\E^{\P^\theta}[\xi],
    \end{align*}
    Then, the nonlinear expectation $\mathcal{E}_{\alpha,\mathcal{P}}[\cdot]$ satisfies Assumption \ref{assE}. In fact, we have
    \begin{align*}
        \mathcal{E}_{\alpha,\mathcal{P}}[X]-\mathcal{E}_{\alpha,\mathcal{P}}[Y]\leq \sup_{\theta\in \Theta}\E^{\P^{\theta}}[Y-X]\leq \sup_{\theta\in \Theta_\kappa}\E^{\P^{\theta}}[Y-X]=\mathcal{E}^{\kappa}[Y-X]\leq G^\kappa_{0,T}[Y-X].
    \end{align*}
\end{example}

\begin{remark}\label{remarkE}


   \noindent (i) By Proposition \ref{g-exp}, for any $X_1,X_2\in L^2(\mathcal{F}_T)$, we have 
   \begin{align*}
       \mathcal{E}[X_1]-\mathcal{E}[X_2]\leq -MG^{-\kappa}_{0,T}[X_2-X_1].
   \end{align*}
   Consequently, we have
        \begin{align}\label{dominate}
           MG^{-\kappa}_{0,T}[X_1-X_2] \leq \mathcal{E}[X_1]-\mathcal{E}[X_2]\leq MG^{\kappa}_{0,T}[X_1-X_2].
        \end{align}
        and 
        \begin{align}\label{dominate'}
           |\mathcal{E}[X_1]-\mathcal{E}[X_2]|\leq MG^{\kappa}_{0,T}[|X_1-X_2|].
        \end{align}

 \noindent (ii) Suppose that $Y-X\geq C$, a.s., where $C>0$ is a constant. Then, we have
 \begin{align*}
     \mathcal{E}[Y]-\mathcal{E}[X]\geq M G^{-\kappa}_{0,T}[Y-X]\geq MG^{-\kappa}_{0,T}[C]= MC e^{-\kappa T},
 \end{align*}
 where we have used Remark \ref{g-exp'}. In this case, the strict comparison property holds, i.e., $\mathcal{E}[Y]>\mathcal{E}[X]$.

   \noindent (iii) Actually, Liu and Wang \cite{LW} and He and Li \cite{HL} also considered the BSDE with nonlinear expectation reflection within the $G$-expectation framework. However, the results cannot cover those in the present paper. 

    \noindent (iv) Without loss of generality, in the following of this paper, we assume that $M=1$. 

    \noindent (v) If the nonlinear expectation degenerates into the classical expectation, the BSDE with nonlinear expectation reflection \eqref{nonlinearyz} turns into the mean reflected BSDE studied in \cite{BEH}.
\end{remark}

\subsection{Existence and uniqueness result}

Motivated by the construction for the solution to the mean reflected BSDE in \cite{BEH}, we define the operator $L_t$ is follows
\begin{align}\label{operator}
L_t: L^2(\mathcal{F}_T)\rightarrow [0,\infty), \ X\mapsto \inf\{x\geq 0:\mathcal{E}[l(t,x+X)]\geq 0\}.
\end{align}
Before investigating the properties of this operator, we first show that it is well-defined, which is a direct consequence of the following lemma.

\begin{lemma}\label{lemma1}
    Suppose that Assumptions \ref{ass2} and \ref{assE} holds. For any fixed $t\in[0,T]$ and $X\in L^2(\mathcal{F}_T)$, define $h(x):=\mathcal{E}[l(t,x+X)]$. Then, the inverse map $h^{-1}$ exists. 
\end{lemma}

\begin{proof}
    For any $x_1<x_2$, we have
    \begin{align*}
        l(t,x_2+X)-l(t,x_1+X)\geq c_l (x_2-x_1).
    \end{align*}
    By Remark \ref{remarkE} (ii), we have $h(x_1)<h(x_2)$, i.e., the function $h$ is strictly increasing. For any $x\geq 0$, note that 
    \begin{align*}
        l(t,x+X)\geq l(t,X)+c_l x, \ l(t,X-x)\leq l(t,x)-c_lx.
    \end{align*}
    It follows from the monotonicity of $\mathcal{E}$, Eq. \eqref{dominate} and Remark \ref{g-exp'} that 
    \begin{align*}
        &\mathcal{E}[l(t,x+X)]\geq \mathcal{E}[l(t,X)+c_l x]\geq \mathcal{E}[l(t,X)]+G^{-\kappa}_{0,T}[c_l x]=\mathcal{E}[l(t,X)]+c_l x e^{-\kappa T},\\
        &\mathcal{E}[l(t,X-x)]\leq \mathcal{E}[l(t,X)-c_l x]\leq \mathcal{E}[l(t,X)]+G^{\kappa}_{0,T}[-c_l x]=\mathcal{E}[l(t,X)]-c_l x e^{-\kappa T}.
    \end{align*}
    Then, we have 
    \begin{align*}
        &\lim_{x\rightarrow+\infty}h(x)\geq \mathcal{E}[l(t,X)]+\lim_{x\rightarrow+\infty}c_l x e^{-\kappa T}=+\infty,\\
        &\lim_{x\rightarrow+\infty}h(-x)\leq \mathcal{E}[l(t,X)]-\lim_{x\rightarrow+\infty}c_l x e^{-\kappa T}=-\infty.
    \end{align*}
    It remains to show that $h$ is continuous. By Proposition \ref{g-exp} and Eq. \eqref{dominate'}, we obtain that 
    \begin{align*}
        \lim_{y\rightarrow x}|\mathcal{E}[l(t,y+X)]-\mathcal{E}[l(t,x+X)]|^2 
        &\leq C_{\kappa, T}\lim_{y\rightarrow x}\E[|l(t,y+X)-l(t,x+X)|^2]=0,
    \end{align*}
   where we have used the dominated convergence theorem for the last equality. The proof is complete.
\end{proof}

The following proposition shows that for any $t\in[0,T]$, the operator $L_t$ is Lipschitz continuous with respect to the $g$-expectation $G^{\kappa}_{0,T}$, which is weaker than Assumption ($H_L$) in \cite{BEH}. 
\begin{proposition}\label{lip}
        For any $X,Y\in L^2(\mathcal{F}_T)$, we have 
        \begin{align*}
            |L_t(X)-L_t(Y)|\leq \frac{C_l}{c_l}e^{\kappa T}G^\kappa_{0,T}\left[|X-Y|\right], \ t\in[0,T].
        \end{align*}
\end{proposition}

\begin{proof}
    Since $l$ is increasing and bi-Lipschitz in its last component, it is easy to check that 
    \begin{align*}
        &l\left(t,L_t(X)+Y+\frac{C_l}{c_l}e^{\kappa T}G^\kappa_{0,T}[|X-Y|]\right)\\
        \geq &c_l \frac{C_l}{c_l}e^{\kappa T}G^\kappa_{0,T}[|X-Y|]+l(t,L_t(X)+Y)\\
        \geq &C_le^{\kappa T}G^\kappa_{0,T}[|X-Y|]+l(t,L_t(X)+X)-C_l|X-Y|.
    \end{align*}
    Taking nonlinear expectations on both sides and applying Eq. \eqref{dominate}, Proposition \ref{g-exp} and Remark \ref{g-exp'} yield that 
    \begin{align*}
        &\mathcal{E}\left[l\left(t,L_t(X)+Y+\frac{C_l}{c_l}G^\kappa_{0,T}[|X-Y|]\right)\right]\\
        \geq &\mathcal{E}\left[C_le^{\kappa T}G^\kappa_{0,T}[|X-Y|]+l(t,L_t(X)+X)-C_l|X-Y|\right]\\
        \geq &G^{-\kappa}_{0,T}\left[C_le^{\kappa T}G_{0,T}^\kappa[|X-Y|]\right]+\mathcal{E}[l(t,L_t(X)+X)-C_l|X-Y|]\\
        \geq &C_lG_{0,T}^\kappa[|X-Y|]+\mathcal{E}[l(t,L_t(X)+X)]-G_{0,T}^\kappa[C_l|X-Y|]\\
        =&\mathcal{E}[l(t,L_t(X)+X)]\geq 0.
    \end{align*}
    By the definition of the operator $L_t$, we have 
    \begin{align*}
        L_t(Y)\leq L_t(X)+\frac{C_l}{c_l}e^{\kappa T}G_{0,T}^\kappa[|X-Y|].
    \end{align*}
    By symmetry of $X$ and $Y$, we obtain the desired result.
\end{proof}

In the following, we show that the operator $\{L_t(\cdot)\}_{t\in[0,T]}$ is continuous in $t$.
\begin{proposition}\label{continuity}
    Suppose that $S\in \mathcal{S}^2$. Then, the function $\{L_t(S_t)\}_{t\in[0,T]}$ is continuous.
\end{proposition}

\begin{proof}
    By a similar analysis as the proof of Lemma \ref{lemma1}, for any constant $a$, the function $\{\mathcal{E}[l(t,S_t+a)]\}_{t\in[0,T]}$ is continuous in $t$. Then, consider the following two cases. 

    Case 1. $\mathcal{E}[l(t,S_t)]>0$. In this case, $L_t(S_t)=0$. Note that
    \begin{align*}
        \lim_{s\rightarrow t}\mathcal{E}[l(s,S_s)]=\mathcal{E}[l(t,S_t)]>0.
    \end{align*}
    It follows that when $|s-t|$ is small enough, $\mathcal{E}[l(s,S_s)]>0$, which implies that $L_s(S_s)=0$.

    Case 2. $\mathcal{E}[l(t,S_t)]\leq  0$. In this case, we have $\mathcal{E}[l(t,S_t+L_t(S_t))]=0$. Besides, for any $x<L_t(S_t)<y$,  we have
    \begin{align*}
        \lim_{s\rightarrow t}\mathcal{E}[l(s,S_s+x)]&=\mathcal{E}[l(t,S_t+x)]<0=\mathcal{E}[l(t,S_t+L_t(S_t))]\\
        &<\mathcal{E}[l(t,S_t+y)]=\lim_{s\rightarrow t}\mathcal{E}[l(s,S_s+y)],
    \end{align*}
    which yields that $\mathcal{E}[l(s,S_s+x)]<0<\mathcal{E}[l(s,S_s+y)]$ when $|s-t|$ is small enough. Consequently, we have $x\leq L_s(S_s)\leq y$. The proof is complete.
\end{proof}

Now, we first establish the well-posedness of the BSDE with nonlinear expectation reflection \eqref{nonlinearyz} when the driver $f$ does not depend on $(Y,Z)$. 

\begin{proposition}\label{prop7}
Suppose that $l,\mathcal{E}$ satisfy Assumption \ref{ass2} and Assumption \ref{assE}. Let $\xi\in L^2(\mathcal{F}_T)$ be such that $\mathcal{E}[l(T,\xi)]\geq 0$. Given $C\in \mathcal{H}^2$, the BSDE with nonlinear expectation reflection
\begin{equation}\label{nonlinearnoyz}
\begin{cases}
Y_t=\xi+\int_t^T C_sds-\int_t^T Z_s dB_s+K_T-K_t, \\
\mathcal{E}[l(t,Y_t)]\geq 0 \textrm{ and }
\int_0^T \mathcal{E}[l(t,Y_t)]dK_t=0,
\end{cases}
\end{equation}
has a unique solution $(Y,Z,K)\in \mathcal{S}^2\times \mathcal{H}^2\times I[0,T]$.
\end{proposition}
\begin{proof}
The proof is similar with the one for Proposition 7 in \cite{BEH}. For readers' convenience, we give a short proof here.

\textbf{Step 1. Uniqueness.} Suppose that $(Y^i,Z^i,K^i)$ are two solutions to \eqref{nonlinearnoyz}, $i=1,2$. Note that $(Y^i+K^i,Z^i)$ may be seen as the solution to the BSDE with terminal value $\xi+K^i_T$ and driver $C$. By the uniqueness of solutions to BSDEs, it suffices to prove that $K^1\equiv K^2$. Otherwise, suppose that there exists $t_1<T$ such that
\begin{displaymath}
K^1_T-K^1_{t_1}>K^2_T-K^2_{t_1}.
\end{displaymath} 
We define
\begin{displaymath}
t_2:=\inf\{t\geq t_1: K_T^1-K^1_t=K_T^2-K^2_t\}.
\end{displaymath}
Then, we have
\begin{displaymath}
K_T^1-K^1_t> K^2_T-K^2_t, \ t_1\leq t<t_2.
\end{displaymath}
Set $X_t=\E_t\left[\xi+\int_t^T C_sds\right]$. Since $K^i$ are deterministic, we have $Y_t^i=X_t+K^i_T-K^i_t$, $i=1,2$. By a similar analysis as in the proof of Lemma \ref{lemma1} for the strict comparison property, for any $t_1\leq t<t_2$, we obtain that
\begin{align*}
\mathcal{E}[l(t,X_t+K_T^1-K_t^1)]>\mathcal{E}[l(t,X_t+K_T^2-K_t^2)]\geq 0.
\end{align*}
The Skorokhod condition implies that $dK^{1}_t=0$ on the interval $[t_1,t_2]$. We deduce that
\begin{align*}
K_T^1-K_{t_2}^1=K_T^1-K_{t_1}^1>K_T^2-K_{t_1}^2\geq K^2_T-K^2_{t_2},
\end{align*}
which contradicts the definition of $t_2$. Therefore, we have $K^1\equiv K^2$.

\textbf{Step 2. Existence.} Recalling that $X_t=\E_t\left[\xi+\int_t^T C_sds\right]$, for any $t\in[0,T]$, we define
\begin{equation*}
    K_t:=\sup_{s\in[0,T]}L_s(X_s)-\sup_{s\in[t,T]}L_s(X_s).
\end{equation*}
By Proposition \ref{continuity}, it is easy to check that $K\in I[0,T]$. Let $(\tilde{Y},Z)$ be the solution to the BSDE with terminal value $\xi$ and driver $C$. Set $Y_t=\tilde{Y}_t+K_T-K_t$, $t\in[0,T]$. We claim that $(Y,Z,K)$ is the solution to \eqref{nonlinearnoyz}. First, it is easy to check that 
\begin{align*}
    Y_t=\xi+\int_t^T C_s ds-\int_t^T Z_s dB_s +K_T-K_t=X_t+K_T-K_t, \ t\in[0,T].
\end{align*}
Second, note that
\begin{equation}\label{representationK}
    K_T-K_t=\sup_{s\in[t,T]}L_s(X_s), \ t\in[0,T].
\end{equation}
Then, we obtain that 
\begin{align*}
    \mathcal{E}[l(t,Y_t)]=\mathcal{E}[l(t,X_t+K_T-K_t)]\geq \mathcal{E}[l(t,X_t+L_t(X_t))]\geq 0.
\end{align*}
Finally, by the definition of $K$, we have 
\begin{itemize}
    \item $\sup_{s\in[t,T]}L_s(X_s)=L_t(X_t)$, $dK_t$-a.e.
    \item $I_{\{L_t(X_t)=0\}}=0$, $dK_t$-a.e.
    \item If $L_t(X_t)>0$, then $\mathcal{E}[l(t,X_t+L_t(X_t))]=0$.
\end{itemize}
It follows that 
\begin{align*}
 \int_0^T \mathcal{E}[l(t,Y_t)]dK_t=\int_0^T \mathcal{E}[l(t,X_t+L_t(X_t))]dK_t=\int_0^T \mathcal{E}[l(t,X_t+L_t(X_t))]I_{\{L_t(X_t)>0\}}dK_t =0. 
\end{align*}
The proof is complete.
\end{proof}

We are now in a position to state the main result of this section.

\begin{theorem}\label{main}
Given $\xi\in L^2(\mathcal{F}_T)$, suppose that $l$ satisfy Assumption \ref{ass2} with $\mathcal{E}[l(T,\xi)]\geq 0$, $f$ and $\mathcal{E}$ satisfy Assumption \ref{assf} and Assumption \ref{assE}. Then, the BSDE with nonlinear expectation reflection \eqref{nonlinearyz} has a unique solution $(Y,Z,K)\in \mathcal{S}^2\times \mathcal{H}^2\times I[0,T]$.
\end{theorem}

\begin{proof}
We first proof that there exist a constant $\delta$ depending on $c_l,C_l,\lambda,\kappa,T$, such that for any $h\in(0,\delta]$, the BSDE with nonlinear expectation reflection \eqref{nonlinearyz} has a unique solution $(Y,Z,K)$ on the time interval $[T-h,T]$. Given $(U,V)\in \mathcal{S}^2_{[T-h,T]}\times\mathcal{H}^2_{[T-h,T]}$, consider the following BSDE with nonlinear expectation reflection 
\begin{displaymath}
\begin{cases}
Y_t=\xi+\int_t^T f(s,U_s,V_s)ds-\int_t^T Z_s dB_s+K_T-K_t, \\
\mathcal{E}[l(t,Y_t)]\geq 0, t\in[T-h,T], \\
K\in I[T-h,T] \textrm{ and } \int_{T-h}^T \mathcal{E}[l(t,Y_t)]dK_t=0.
\end{cases}
\end{displaymath}
By a similar analysis as the proof of Proposition \ref{prop7}, the above reflected BSDE admits a unique solution  $(Y,Z,K)\in \mathcal{S}^2_{[T-h,T]}\times \mathcal{H}^2_{[T-h,T]}\times I[T-h,T]$. We define a mapping $\Gamma:\mathcal{S}^2_{[T-h,T]}\times\mathcal{H}^2_{[T-h,T]}\rightarrow \mathcal{S}^2_{[T-h,T]}\times\mathcal{H}^2_{[T-h,T]}$ by setting
\begin{align*}
    \Gamma(U,V)=(Y,Z).
\end{align*}
We show that $\Gamma$ is a contraction mapping when $\delta$ is sufficiently small. For this purpose, given $(U^i,V^i)\in \mathcal{S}^2_{[T-h,T]}\times\mathcal{H}^2_{[T-h,T]}$, $i=1,2$, set $(Y^i,Z^i)=\Gamma(U^i,V^i)$ and
\begin{align*}
\hat{F}_t=F^1_t-F^2_t, \textrm{ where } F=Y,Z,K,U,V, \ \hat{f}_t=f(t,U^1_t,V^1_t)-f(t,U^2_t,V^2_t). 
\end{align*}
Then, we have  
\begin{align*}
\hat{Y}_t=\E_t\left[\int_t^T \hat{f}_sds\right]+\hat{K}_T-\hat{K}_t.
\end{align*}
Applying the Doob inequality, there exists a constant $M(\lambda)$ depending on $\lambda$, such that 
\begin{equation}\label{haty}
\E\left[\sup_{t\in[T-h,T]}|\hat{Y}_t|^2\right]\leq M(\lambda)\E\left[\left(\int_{T-h}^T \left[|\hat{U}_t|+|\hat{V}_t|\right]dt\right)^2\right]+M(\lambda)\sup_{t\in[T-h,T]}|\hat{K}_T-\hat{K}_t|^2.
\end{equation}
By Proposition \ref{lip} and recalling \eqref{representationK}, we have
\begin{align*}
\sup_{t\in[T-h,T]}|\hat{K}_T-\hat{K}_t|&\leq  2\sup_{t\in[T-h,T]}|L_t(X^1_t)-L_t(X^2_t)|\\
&\leq \frac{2C_l}{c_l}e^{\kappa T}\sup_{t\in[T-h,T]}G^\kappa_{0,T}[|X^1_t-X^2_t|]\\
&\leq \frac{2C_l}{c_l}e^{\kappa T}G_{0,T}^\kappa\left[\sup_{t\in[T-h,T]}\E_t\left[\int_{T-h}^T|\hat{f}_s|ds\right]\right],
\end{align*}
where 
\begin{align*}
X^i_t=\E_t\left[\xi+\int_t^T f(s,U^i_s,V^i_s)ds\right].
\end{align*}
By Proposition \ref{g-exp} and the BDG inequality, there exists a positive constant $M(c_l,C_l,\lambda,\kappa,T)$, such that
\begin{equation}\label{hatk}
\sup_{t\in[0,T]}|\hat{K}_T-\hat{K}_t|^2\leq M(c_l,C_l,\lambda,\kappa,T)\E\left[\left(\int_{T-h}^T \left[|\hat{U}_t|+|\hat{V}_t|\right]dt\right)^2\right].
\end{equation}
Combining Eqs. \eqref{haty} and \eqref{hatk} yields that
\begin{equation}\label{haty'}
\E\left[\sup_{t\in[0,T]}|\hat{Y}_t|^2\right]\leq M(c_l,C_l,\lambda,\kappa,T)\E\left[\left(\int_{T-h}^T \left[|\hat{U}_t|+|\hat{V}_t|\right]dt\right)^2\right].
\end{equation}

On the other hand, note that
\begin{align*}
\int_{T-h}^T \hat{Z}_s dB_s=\int_{T-h}^T\hat{f}_s ds-\hat{Y}_{T-h}+\hat{K}_T-\hat{K}_{T-h}.
\end{align*}
Applying \eqref{hatk}, \eqref{haty'} and the H\"{o}lder inequality yields that
\begin{align*}
\E\left[\int_{T-h}^T |\hat{Z}_s|^2ds\right]\leq M(c_l,C_l,\lambda,\kappa,T)\E\left[\left(\int_{T-h}^T \left[|\hat{U}_t|+|\hat{V}_t|\right]dt\right)^2\right].
\end{align*}
All the above analysis indicates that 
\begin{align*}
&\E\left[\sup_{t\in[T-h,T]}|\hat{Y}_t|^2+\int_{T-h}^T |\hat{Z}_s|^2ds\right]\\
\leq &M(c_l,C_l,\lambda,\kappa,T)\E\left[\left(\int_{T-h}^T \left[|\hat{U}_t|+|\hat{V}_t|\right]dt\right)^2\right]\\
\leq &M(c_l,C_l,\lambda,\kappa,T)h\max(1,h)\E\left[\sup_{t\in[T-h,T]}|\hat{U}_t|^2+\int_{T-h}^T|\hat{V}_s|^2ds\right].
\end{align*}
Let $\delta$ be sufficiently small such that $M(c_l,C_l,\lambda,\kappa,T)\delta\max(1,\delta)<1$. Then $\Gamma$ is a contraction mapping. It follows that the BSDE with nonlinear expectation reflection \eqref{nonlinearyz} has a unique solution on $[T-h,T]$ for any $h\in(0, \delta]$.

For the general case, by a standard BSDE approach, we split the whole interval $[0,T]$ into finitely many small intervals. On each small interval, we can get a local solution. The global solution on the whole time interval can be constructed by stitching the  local ones. Uniqueness is a direct consequence from the uniqueness on each small interval. The proof is complete. 
\end{proof}

\begin{remark}
    Another frequently used method to construct the solution to the reflected problem is approximation via penalization (see \cite{EKPPQ,CK,PX} for the case of classical reflection and \cite{BEH,Li2024} for the case of mean reflection). However, this method is not valid for the case of nonlinear expectation reflection. Consider a simple case that the loss function $l$ is linear, i.e., the constraint is written as
    \begin{align*}
        \mathcal{E}[Y_t]\geq l_t, \ t\in[0,T],
    \end{align*}
    for some $l\in C[0,T]$. For any $n\in\mathbb{N}$, a candidate family of penalized equations should be 
    \begin{align*}
        Y_t^n=\xi+\int_t^T f(s,Y^n_s,Z^n_s)ds+n\int_t^T (\mathcal{E}[Y^n_s]-l_s)^- ds-\int_t^T Z^n_sdB_s,
    \end{align*}
    where we just replace the classical expectation in the penalized mean-field BSDEs in Proposition 6 of \cite{BEH} by our nonlinear expectation $\mathcal{E}$. As far as we know, there is no existing result for the well-posedness of this type of BSDEs. The second problem is that, even if we have established the existence and uniqueness result for the above penalized ``nonlinear expectation BSDEs", since their coefficients are not nondecreasing in the nonlinear expectation of the solutions, a natural conjecture is that we do not have the comparison property of $Y^n$ in $n$ (see Example 3.2 and Theorem 3.2 for the comparison theorem for mean-field BSDEs in \cite{BLP}). Moreover, due to the nonlinearity of $\mathcal{E} $, we will encounter some difficulty (e.g., the Fubini theorem does not hold) when establishing the uniform a priori estimate on the sequence $(Y^n,Z^n,K^n)$, where $K^n:=\int_0^\cdot n((\mathcal{E}[Y^n_s]-l_s)^-)ds$. Therefore, the proof of Proposition 6 in \cite{BEH} is no longer valid for the nonlinear expectation case. 
\end{remark}

\subsection{Properties of solution to BSDE with nonlinear expectation reflection}

In this section, we investigate the comparison theorems for the BSDE with nonlinear expectation reflection. To this end, let us first introduce the following definition. 
\begin{definition}
    Let $\mathcal{E}^i:L^2(\mathcal{F}_T)\rightarrow\mathbb{R}$ be two operators. We call $\mathcal{E}^1$ is stronger than $\mathcal{E}^2$, denoted by $\mathcal{E}^1\geq \mathcal{E}^2$, if for any $X\in L^2(\mathcal{F}_T)$,  we have $\mathcal{E}^1[X]\geq \mathcal{E}[X^2]$. 
\end{definition}

\begin{remark}
 It is natural that $G^{-\kappa}_{0,T}\leq \E\leq G^{\kappa}_{0,T}$. Moreover, recalling Proposition \ref{g-exp}, we have $G^{-\kappa}_{0,T}\leq G_{0,T}\leq G^{\kappa}_{0,T}$. 
\end{remark}
 
We first provide the pointwise comparison property for the solutions to BSDEs with nonlinear expectation reflection. However, recall that Example 3.3 in \cite{HHLLW} indicates that the pointwise comparison theorem for mean reflected BSDEs does not hold in general. Therefore, we can only obtain the desired result under some typical structure of the parameters. Intuitively, the larger the loss function and the stronger the nonlinear expectation, the less the force aiming to push the solution upward. Therefore, it is natural to conjecture that if the loss function is larger and the nonlinear expectation is stronger, then the corresponding solution is smaller.
\begin{proposition}\label{prop11}
Suppose that $l^i$, $\mathcal{E}^i$, $i=1,2$, satisfy Assumption \ref{ass2} and Assumption \ref{assE} and the coefficient $f$ satisfying Assumption \ref{assf} takes the following form
\begin{align}\label{equation24}
f:(t,y,z)\mapsto a_t y+h(t,z),
\end{align}
where $a$ is a deterministic and bounded measurable function.  Given $\xi\in L^2(\mathcal{F}_T)$, $c^1,c^2\in\mathbb{R}$ with $c^1\geq c^2$ and $\mathcal{E}^i[l^i(T,\xi+c^i)]\geq 0$, $i=1,2$. 
Let $(Y^i,Z^i,K^i)$ be the solution to BSDE with nonlinear expectation reflection with parameters $(\xi+c^i,f,l^i,\mathcal{E}^i)$. If for any $(t,x)\in[0,T]\times\mathbb{R}$, $l^1(t,x)\leq l^2(t,x)$, a.s. and $\mathcal{E}^2$ is stronger than $\mathcal{E}^1$, then, for any $t\in[0,T]$, we have $Y^1_t\geq  Y^2_t$.
\end{proposition}

\begin{proof}
Let $(Y,Z,K)$ be the solution to the nonlinear expectation reflected BSDE with parameters $(\xi,f,l,\mathcal{E})$. We define $A_t:=\int_0^t a_s ds$, $t\in[0,T]$. Then, set
\begin{align*}
\tilde{Y}_t=e^{A_t} Y_t,  \ \tilde{Z}_t=e^{A_t} Z_t, \ \tilde{K}_t=\int_0^t e^{A_s} dK_s. 
\end{align*}
It is easy to check that $(\tilde{Y},\tilde{Z},\tilde{K})$ is the solution to the nonlinear expectation reflected BSDE with parameters $(\tilde{\xi},\tilde{f},\tilde{l},\mathcal{E})$, where 
\begin{align*}
\tilde{\xi}=e^{A_T}\xi, \ \tilde{f}(t,z)=e^{A_t}h(t,e^{-A_t}z), \ \tilde{l}(t,y)=l(t,e^{-A_t}y).
\end{align*}
Therefore, it suffices to prove the result for the case where $f$ does not depends on $y$. In this special case, note that $(Y^i-c^i-(K^i_T-K^i),Z^i)$, $i=1,2$,  are solutions the to BSDE with terminal value $\xi$ and coefficient $f$. It follows from the uniqueness result for BSDEs  that
\begin{align}\label{equation25}
Y^1_t-c^1-(K^1_T-K^1_t)=Y^2_t-c^2-(K^2_T-K^2_t), \ t\in[0,T]. 
\end{align} 
To obtain the desired result, we only need to show that $c^1+K^1_T-K^1_t\geq c^2+K^2_T-K^2_t$, for any $t\in[0,T]$. We prove this fact by contradiction. Suppose that there exists some $t_1<T$, such that 
\begin{align*}
c^1+K^1_T-K^1_{t_1}<c^2+ K^2_T-K^2_{t_1}.
\end{align*}
Set 
\begin{align*}
t_2=\inf\{t\geq t_1: c^2+K^2_T-K^2_t\leq c^1+{K}^1_T-{K}^1_t\}.
\end{align*}
Noting that $c^1\geq c^2$, it follows that $t_2\leq T$. Since $K^i$, $i=1,2$, are continuous, we have 
\begin{align*}
c^2+K^1_T-K^1_{t_2}= c^1+{K}^2_T-{K}^2_{t_2}, \ c^2+K^2_T-K^2_t> c^1+{K}^1_T-{K}^1_t, \ t\in [t_1,t_2),
\end{align*}
which together with Eq. \eqref{equation25} yields that 
\begin{align*}
Y^2_t-{Y}^1_t=c^2+(K^2_T-K^2_t)-c^1-(K^1_T-K^1_t)>0, \ t\in[t_1,t_2).
\end{align*}
By a similar analysis as in the proof of Lemma \ref{lemma1}, for any $t\in[t_1,t_2)$, we have
\begin{align*}
\mathcal{E}^2[l(t,Y^2_t)]>\mathcal{E}^2[l(t,Y^1_t)]\geq \mathcal{E}^1[l(t,{Y}^1_t)]\geq 0.
\end{align*}
It follows from the Skorokhod condition that $d K^2_t=0$, $t\in[t_1,t_2)$. All the above analysis indicates that 
\begin{align*}
c^1+K^1_T-K^1_{t_1}<c^2+K^2_T-K^2_{t_1}=c^2+K^2_T-K^2_{t_2}=c^1+K^1_T-K^1_{t_2}\leq c^1+ K^1_T-K^1_{t_1},
\end{align*}
which is a contradiction. The proof is complete.
\end{proof}

The Skorokhod condition usually ensures the minimality of the solution for the reflected BSDEs. For the mean reflected BSDE investigated in \cite{BEH}, due to the fact that the constraint is made in expectation but not pointwisely, the minimality property needs some additional assumption for the coefficient $f$. Mimicking the proof of Theorem 11 in \cite{BEH}, we can also obtain the following comparison theorem for BSDE with nonlinear expectation reflection. 
\begin{proposition}\label{prop3.9}
    Suppose that the coefficient $f$ satisfies Assumption \ref{assf} and takes the form of \eqref{equation24}. Let Assumptions \ref{ass2} and \ref{assE} holds. Given $\xi\in L^2(\mathcal{F}_T)$ with $\mathcal{E}[l(T,\xi)]\geq 0$.  Let $(Y,Z,K)$ be the solution to the nonlinear expectation reflected BSDE with parameters $(\xi,f,l,\mathcal{E})$ and let $(Y',Z',K')\in \mathcal{S}^2\times\mathcal{H}^2\times I[0,T]$ satisfy the following equation
   \begin{equation*}
\begin{cases}
Y'_t=\xi+\int_t^T f(s,Y'_s,Z'_s)ds-\int_t^T Z'_s dB_s+K'_T-K'_t, \\
\mathcal{E}[l(t,Y'_t)]\geq 0.
\end{cases}
\end{equation*}
Then, we have $Y_t\leq Y'_t$, $t\in[0,T]$.
\end{proposition}

\begin{remark}\label{remarkminimal}
   It should be pointing out that we cannot derive the minimality of the solution to the BSDE with  nonlinear expectation reflection if $K'$ is not assumed to be determinisitc in Proposition \ref{prop3.9}. For readers' convenience, we provide a counterexample below, which can also be found in \cite{BEH}. Suppose that $\mathcal{E}$ degenerates into the classical expectation, $l(t,x)=x-u$ with $u$ being a constant and the driver is a positive constant, denoted by $\gamma$. Given $\xi\in L^2(\mathcal{F}_T)$ with $u<\E[\xi]<u+\gamma T$, let $t^*$ be such that $\E[\xi]-\gamma(T-t^*)=u$. We may obtain that the first component of the solution to the mean reflected BSDE with parameters $(\xi,f,l)$ is as follows
   \begin{align*}
       Y_t=\E[\xi|\mathcal{F}_t]-\gamma(T-t)+K_T-K_t,
   \end{align*}
   where $K_t=\gamma(t\wedge t^*)$. For any constant $\alpha$, set 
   \begin{align*}
       M^\alpha_t=\exp\left(\alpha B_t-\frac{1}{2}\alpha^2 t\right), \ K^\alpha_t=\int_0^t M^\alpha_sdK_s.
   \end{align*}
   Let $(Y^\alpha,Z^\alpha)$ be the solution to the following BSDE
   \begin{align*}
       Y^\alpha_t=\xi-\int_t^T\gamma dt-\int_t^T Z^\alpha_sdB_s+K^\alpha_T-K^\alpha_t.
   \end{align*}
   It is easy to check that
   \begin{align*}
       Y^\alpha_t=\E[\xi|\mathcal{F}_t]-\gamma(T-t)+M^\alpha_t(K_T-K_t),
   \end{align*}
   and $Y^\alpha$ satisfies the constraint, i.e., $\E[Y^\alpha_t]=\E[Y_t]\geq u$. However, we do not have $Y_t\leq Y^\alpha_t$.
\end{remark}


In the following, we investigate some weaker comparison properties for BSDE with nonlinear expectation reflections using the representation for its solution. Actually, the expectation of the solution to a BSDE with double mean reflections corresponds to a certain optimization problem (see \cite{FS,Li2024}). We then establish the similar result for the case of nonlinear expectation reflection. More precisely, let $(Y,Z,K)$ be the solution to the BSDE with nonlinear expectation reflection \eqref{nonlinearyz}. We define
\begin{align*}
\bar{Y}_t:=\E_t\left[\xi+\int_t^T f(s,Y_s,Z_s)ds\right].
\end{align*}
It is easy to check that $Y_t=\bar{Y}_t-\E[\bar{Y}_t]+\E[Y_t]$, $t\in[0,T]$. 
Consider the following equation
\begin{align*}
    \mathcal{E}[l(t,\bar{Y}_t-\E[\bar{Y}_t]+x)]=0.
\end{align*}
By Lemma \ref{lemma1}, the above equation admits a unique solution, which is denoted by $\bar{l}_t$. 

\begin{theorem}\label{theorem3.6}
Suppose that $(Y,Z,K)$ is the solution to the BSDE with nonlinear expectation reflection \eqref{nonlinearyz}. Then, for  any $t\in[0,T]$, we have
\begin{align*}
\E[Y_t]=\sup_{s\in[t,T]}\left\{\E\left[\int_t^{s} f(u,Y_u,Z_u)du+\xi I_{\{s=T\}}\right]+\bar{l}_s I_{\{s<T\}}\right\}.
\end{align*}
\end{theorem}

\begin{proof}
The result clearly holds for $t=T$. In the following, we only consider the case that $t<T$. First, we show that for any $0\leq t\leq s\leq T$, 
\begin{align*}
    \E[Y_t]\geq \E\left[\int_t^{s} f(u,Y_u,Z_u)du+\xi I_{\{s=T\}}\right]+\bar{l}_s I_{\{s<T\}}.
\end{align*}
Since $K_T-K_t\geq 0$ for any $t\in[0,T]$, the above inequality clearly holds for $s=T$. For the case that $0\leq t\leq s<T$, we first claim that $\E[Y_r]\geq \bar{l}_r$, $r\in[0,T]$. In fact, it is a direct consequence of the following inequality
\begin{align*}
    \mathcal{E}\left[l\left(r,\bar{Y}_r-\E[\bar{Y}_r]+\bar{l}_r\right)\right]=0\leq \mathcal{E}[l(r,Y_r)]=\mathcal{E}\left[l\left(r,\bar{Y}_r-\E[\bar{Y}_r]+\E[Y_r]\right)\right].
\end{align*}
Note that for $0\leq t\leq s<T$,
\begin{align*}
    Y_t=Y_s+\int_t^{s} f(u,Y_u,Z_u)du-\int_t^{s} Z_udB_u+K_s-K_t.
\end{align*}
Taking expectations on both sides yields that
\begin{align*}
    \E[Y_t]=\E\left[Y_s+\int_t^{s} f(u,Y_u,Z_u)du\right]+K_s-K_t\geq \E\left[\int_t^{s} f(u,Y_u,Z_u)du\right]+\bar{l}_s.
\end{align*}
All the above analysis indicates that 
\begin{align*}
\E[Y_t]\geq \sup_{s\in[t,T]}\left\{\E\left[\int_t^{s} f(u,Y_u,Z_u)du+\xi I_{\{s=T\}}\right]+\bar{l}_s I_{\{s<T\}}\right\}.
\end{align*}

We define
\begin{align*}
    s^*:=\inf\{s\geq t: \E[Y_s]\leq \bar{l}_s\}\wedge T.
\end{align*}
It remains to prove that for $t\in[0,T)$, 
\begin{align*}
   \E[Y_t]=\E\left[\int_t^{s^*} f(u,Y_u,Z_u)du+\xi I_{\{s^*=T\}}\right]+\bar{l}_{s^*} I_{\{s^*<T\}}. 
\end{align*}

\textbf{Case 1. $s^*=t$.} In this case, we have $\E[Y_t]=\bar{l}_t=\bar{l}_{s^*}$.

\textbf{Case 2. $t<s^*\leq T$.} Then, for any $r\in[t,s^*)$, we have $\E[Y_r]>\bar{l}_r$. It is easy to check that 
\begin{align*}
    \mathcal{E}[l(r,Y_r)]=\mathcal{E}\left[l\left(r,\bar{Y}_r-\E[\bar{Y}_r]+\E[Y_r]\right)\right]>\mathcal{E}\left[l\left(r,\bar{Y}_r-\E[\bar{Y}_r]+\bar{l}_r\right)\right]=0,
\end{align*}
which together with the Skorokhod condition indicates that for $r\in[t,s^*)$, $dK_r=0$. It follows from the continuity of $K$ that $K_t=K_{s^*}$. Therefore, we have 
\begin{align*}
    Y_t=Y_{s^*}+\int_t^{s^*} f(u,Y_u,Z_u)du-\int_t^{s^*}Z_udB_u.
\end{align*}
Taking expectations on both sides yields the desired result. The proof is complete.
\end{proof}

Now, we are ready to establish the comparison properties for the expectation of the solutions to the BSDEs with nonlinear expectation reflections using Theorem \ref{theorem3.6}.

\begin{corollary}\label{cor3.2'}
   Suppose that $l^i$ are deterministic functions satisfying Assumption \ref{ass2}, $i=1,2$. Let $(Y^i,Z^i,K^i)$ be the solution to the following mean reflected BSDEs
    \begin{displaymath}
\begin{cases}
Y^i_t=\xi^i+\int_t^T (a_sY^i_s+f^i_s) ds-\int_t^T Z^i_s dB_s+K^i_T-K^i_t, \\
\E[l^i(t,Y^i_t)]\geq 0, \\
K^i\in I[0,T] \textrm{ and }\int_0^T \E[l(t,Y^i_t)]dK^i_t=0,
\end{cases}
\end{displaymath}
where $\{a_t\}_{t\in[0,T]}$ is a deterministic and bounded measurable function,  $f^i\in \mathcal{H}^2$ and $\xi^i\in L^2(\mathcal{F}_T)$ satisfy $\E[l^i(T,\xi^i)]\geq 0$, $i=1,2$. Suppose that 
\begin{itemize}
    \item $\E[\xi^1]\geq \E[\xi^2]$, $\E[f^1_t]\geq \E[f^2_t]$, for any $t\in[0,T]$.
    \item For any $t\in[0,T]$, $l^1(t,\cdot)$ is concave and $l^2(t,\cdot)$ is convex. 
    \item For any $(t,x)\in[0,T]\times\mathbb{R}$, $l^1(t,x)\leq l^2(t,x)$. 
\end{itemize}Then, we have $\E[Y^1_t]\geq \E[Y^2_t]$, for any $t\in[0,T]$. 
\end{corollary}

\begin{proof}
    Similar as the proof of Proposition \ref{prop11}, we only consider the case that $a\equiv 0$. Set 
    \begin{align*}
\bar{Y}^i_t=\E_t\left[\xi^i+\int_t^T f^i_s ds\right].
\end{align*}
    By Theorem \ref{theorem3.6}, we have 
    \begin{align*}
\E[Y^i_t]=\sup_{s\in[t,T]}\left\{\E\left[\int_t^{s} f^i_u du+\xi^i I_{\{s=T\}}\right]+\bar{l}^i_s I_{\{s<T\}}\right\},
\end{align*}
where for any $t\in[0,T]$, $\bar{l}^i_t$ is the solution to the following equation
\begin{align*}
    \E[l^i(t,\bar{Y}^i_t-\E[\bar{Y}^i_t]+x)]=0.
\end{align*}
It remains to prove that for any $t\in[0,T]$, $\bar{l}^1_t\geq \bar{l}^2_t$. Let $(l^i)^{-1}(t,\cdot)$ be the inverse map of $l^i(t,\cdot)$.  Since $l^1(t,\cdot)$ is concave, we have
\begin{align*}
    l^1(t,\bar{l}^1_t)=l^1(t,\E[\bar{Y}^1_t]-\E[\bar{Y}^1_t]+\bar{l}^1_t)\geq \E[l^1(t,\bar{Y}^1_t-\E[\bar{Y}^1_t]+\bar{l}^1_t)]=0,
\end{align*}
which indicates that $\bar{l}^1_t\geq (l^1)^{-1}(t,0)$. Similarly, we have $\bar{l}^2_t\leq (l^2)^{-1}(t,0)$. It is easy to check that 
\begin{align*}
   (l^1)^{-1}(t,0) \geq (l^2)^{-1}(t,0).
\end{align*}
The proof is complete.
\end{proof}

\begin{remark}
 (i) Compared with Theorem 3.2 case 2 in \cite{HHLLW}, in Proposition \ref{prop11}, our constraint is made by a nonlinear expectation and our loss functions may be different. Besides, Theorem 3.2 in \cite{HHLLW} needs the assumption that $C\leq 1$, where $C$ is the Lipschitz constant for the following equation
\begin{align*}
            |L_t(X)-L_t(Y)|\leq C\E[|X-Y|], \ t\in[0,T].
        \end{align*}
When the loss function $l$ is bi-Lipschitz continuous, this fact amounts to say that it is linear. In our situation, the requirement for the continuity of the operator $L_t(\cdot)$ is weaker (see Proposition \ref{lip}) and our loss functions can be nonlinear. It is worth pointing out that although our assumptions are weaker, the result in Proposition \ref{prop11} is stronger, i.e., we may obtain the comparison property for the pointwise value of $Y$ but not restricted to its expectation.

\noindent (ii) Corollary \ref{cor3.2'} extends Theorem 3.2 case 1 in \cite{HHLLW}. In fact, we remark again that under the assumption $C\leq 1$ in \cite{HHLLW}, the loss function almost needs to be linear. However, in Corollary \ref{cor3.2'}, the loss functions may be nonlinear. Besides, we do not need to assume that $l^1\equiv l^2$ as in Theorem 3.2 in \cite{HHLLW}. 
\end{remark}

\section{Application to superhedging under risk constraint}


Recall that in \cite{BEH}, the BSDE with mean reflection can be applied to the superhedging problem under running risk management constraint. More precisely, the first component of the solution can be seen as the value of a portfolio. From a financial point of view, it is reasonable to require that at each time $t$, the value of the portfolio remains acceptable. As shown in \cite{ADEH,Delbaen,FS02}, the acceptance set can be characterized by the so-called risk measures, whose definition is given below. 

\begin{definition}
    A functional $\rho:L^2(\mathcal{F}_T)\rightarrow\mathbb{R}$ is called a risk measure if it satisfies $\rho(0)=0$ and 
    \begin{itemize}
    \item[(a)] Monotonicity: $X\geq Y\Rightarrow \rho(X)\leq \rho(Y)$;
    \item[(b)] Translation invariance: $\rho(X+a)=\rho(X)-a$, for any $a\in\mathbb{R}$.
\end{itemize}
    $\rho$ is called a coherent risk measure if it satisfies   (a), (b) and 
    \begin{itemize}
    \item[(c)] Sub-additivity: $\rho(X+Y)\leq \rho(X)+\rho(Y)$;
    \item[(d)] Positive homogeneity: $\rho(aX)=a\rho(X)$, for any $a\geq 0$.
\end{itemize}
    $\rho$ is called a convex risk measure if it satisfies (a)-(b) and
    \begin{itemize}
        \item[(e)] Convexity: $\rho(\lambda X+(1-\lambda)Y)\leq \lambda \rho(X)+(1-\lambda)Y$, for any $\lambda\in[0,1]$. 
    \end{itemize}
\end{definition}

\begin{remark}
    By the results in \cite{Delbaen,FS02,FG}, if $\rho$ is a coherent risk measure, there exists a set $\mathcal{P}$ of $\P$-absolutely continuous probability measures such that 
    \begin{align}\label{coherent}
        \rho(X)=\sup_{\Q\in \mathcal{P}}\E^{\Q}[-X].
    \end{align}
    If $\rho$ a convex risk measure, there exists a convex set of probability measures $\mathcal{P}$ and a convex functional $F:\mathcal{P}\rightarrow \mathbb{R}\cup \{+\infty\}$ such that 
    \begin{align}\label{convex}
        \rho(X)=\sup_{\Q\in \mathcal{P}}\left\{\E^{\Q}[-X]-F(\Q)\right\}.
    \end{align}
\end{remark}

Now, we are given a set of risk measures $\{\rho(t,\cdot)\}_{t\in[0,T]}$ and a time indexed deterministic benchmark $\{q_t\}_{t\in[0,T]}$. Consider a financial market consists of a riskless bond with price $S^0$ evolving according to the following equation
\begin{align*}
    dS^0_t=r S^0_t dt,
\end{align*}
where $r>0$ is the interest rate, and a stock $S$ whose dynamics is given by the following stochastic differential equation 
\begin{align*}
    dS_t=S_t(\mu dt+\sigma dB_t),
\end{align*}
where $\mu$ is the appreciation rate and $\sigma$ represents the volatility. The agent with given initial wealth $w$ will invest in the financial market, who chooses a portfolio $\pi_t$ and a consumption plan $C_t$ at time $t$, where $\pi_t$ is the proportion of the wealth $V_t$ invested in the stock and $C_t$ represents the cumulative amount of consumption made before time $t$. Adapted to the framework of this paper, the consumption plan is restricted to be a deterministic function. Then, the wealth process associated to the consumption-investment strategy $(\pi,C)$ can be characterized by the following equation
\begin{align*}
    dV_t=r V_tdt+(\mu-r)\pi_t V_t dt+\sigma \pi_t V_t dB_t-dC_t,  \ V_0=w.
\end{align*}
In order to make sure the wealth is admissible at each time $t$, we propose the following constraint
\begin{align*}
    \rho(t,V_t)\leq q_t, \ t\in[0,T].
\end{align*}
Now, given a contingent claim $\xi\in L^2(\mathcal{F}_T)$, the superhedging price $p$ is defined by $p=\inf_{w\in \mathcal{U}}w$, where 
\begin{align*}
    \mathcal{U}=\left\{w\in\mathbb{R}:\exists (\pi,C) \textrm{ such that } V_T\geq \xi \textrm{ and } \rho(t,V_t)\leq q_t, t\in[0,T]\right\}.
\end{align*}
A natural candidate of the superhedging price is $Y_0$, where $Y$ is the first component of the solution to the following reflected BSDE:
\begin{equation*}
\begin{cases}
Y_t=\xi-\int_t^T(rY_s+\frac{\mu-r}{\sigma}Z_s) ds-\int_t^T Z_s dB_s+K_T-K_t, \\
q_t-\rho(t,Y_t)\geq 0, \\
K\in I[0,T] \textrm{ and }\int_0^T (q_t-\rho(t,Y_t))dK_t=0.
\end{cases}
\end{equation*}

To summarize, the objective of this section is to investigate the following type of reflected BSDE:
\begin{equation}\label{BSDEwithriskmeasure}
\begin{cases}
Y_t=\xi+\int_t^T f(s,Y_s,Z_s)ds-\int_t^T Z_s dB_s+K_T-K_t, \\
q_t-\rho(t,Y_t)\geq 0, \\
K\in I[0,T] \textrm{ and }\int_0^T (q_t-\rho(t,Y_t))dK_t=0,
\end{cases}
\end{equation}
which is called the BSDE with risk measure reflection as in \cite{BEH}. In order to construct the solution to the above equation, similar as the analysis in the previous section, we need to introduce the following operator:
\begin{align}\label{operator'}
\tilde{L}_t: L^2(\mathcal{F}_T)\rightarrow [0,\infty), \ X\mapsto \inf\{x\geq 0:q_t-\rho(t,x+X)\geq 0\}.
\end{align}
The operator can be represented explicitly. In fact, since $\rho(t,\cdot)$ is a risk measure, using the translation invariance property, we have
\begin{align*}
    q_t-\rho(t,x+X)=q_t+x-\rho(t,X).
\end{align*}
Consequently, we obtain that  
\begin{align*}
    \tilde{L}_t(X)=(\rho(t,X)-q_t)^+.
\end{align*}
We make the following assumption for the risk measures $\{\rho(t,\cdot)\}_{t\in[0,T]}$. 

\begin{assumption}\label{assrho}
    For any $X\in L^2(\mathcal{F}_T)$, $\{\rho(t,X)\}_{t\in [0,T]}$ is continuous. For any  $t\in[0,T]$, there exists a positive constant $M$, such that 
    \begin{align*}
        \rho(t,X)-\rho(t,Y)\leq M G_{0,T}^{\kappa}[Y-X], \ X,Y\in L^2(\mathcal{F}_T).
    \end{align*}
\end{assumption}

\begin{remark}
    Similar as Remark \ref{remarkE}, if $\rho$ satisfies Assumption \ref{assrho}, we have
     \begin{align*}
        |\rho(t,X)-\rho(t,Y)|\leq M G_{0,T}^{\kappa}[|X-Y|], \ X,Y\in L^2(\mathcal{F}_T).
    \end{align*}
\end{remark}

\begin{theorem}\label{thm12}
    Let $\{\rho(t,\cdot)\}_{t\in[0,T]}$ be a collection of risk measures satisfying Assumption \ref{assrho}. Suppose that $\{q_t\}_{t\in[0,T]}$ is a continuous function and $f$ satisfies Assumption \ref{assf}. Given $\xi\in L^2(\mathcal{F}_T)$ with $\rho(T,\xi)\leq q_T$, the reflected BSDE \eqref{BSDEwithriskmeasure} admits a unique solution. 
\end{theorem}

\begin{proof}
    The proof is analogous to those for Proposition \ref{prop7} and Theorem \ref{main}. We only need to show that for any $S\in \mathcal{S}^2$, the function $\{\rho(t,S_t)\}_{t\in[0,T]}$ is continuous. In fact, for any $s,t\in[0,T]$, it is easy to check that 
    \begin{align*}
        |\rho(t,S_t)-\rho(s,S_s)|&\leq |\rho(t,S_t)-\rho(s,S_t)|+|\rho(s,S_t)-\rho(s,S_s)|\\
        &\leq |\rho(t,S_t)-\rho(s,S_t)|+M G_{0,T}^{\kappa}[|S_t-S_s|]\\
        &\leq |\rho(t,S_t)-\rho(s,S_t)|+C(\E[|S_t-S_s|^2])^{1/2},
    \end{align*}
    where $C$ is a constant depending on $M,\kappa,T$. Then we get the desired result.
\end{proof}

\begin{remark}\label{remark4.5}
   In order to establish the well-posedness of the BSDE with risk measure reflection, in \cite{BEH}, the risk measure is required to satisfy the following Lipschitz assumption
    \begin{align*}
        |\rho(t,X)-\rho(t,Y)|\leq C\E[|X-Y|], \ t\in[0,T], \ X,Y\in L^2(\mathcal{F}_T).
    \end{align*}
    Consider the coherent risk measure (resp., convex risk measure) $\rho(t,\cdot)$ with representation \eqref{coherent} (resp., \eqref{convex}) associated with probability measures $\mathcal{P}$ (resp.,  with probability measures $\mathcal{P}$ and penalty function $F$). A sufficient condition to make sure the Lipschitz continuity is that the density $\frac{d\Q}{d\P}$ is bounded for any $\Q\in \mathcal{P}$. This condition is somewhat restrictive that excludes some important examples.  
    Let $\mathcal{P}$ be the same as in Example \ref{egofnonlinear expectation} (ii). Then, the coherent risk measure (resp., the convex risk measure) $\rho(t,\cdot)$ induced by $\mathcal{P}$ (resp., by $\mathcal{P}$ and $F$) may not be Lipschitz since the densities are not bounded. However, by Proposition 3.4 in \cite{EPQ}, we have 
    \begin{align*}
        \rho(t,X)-\rho(t,Y)\leq \sup_{\theta\in \Theta}\E^{\P^{\theta}}[Y-X]\leq \sup_{\theta\in \Theta_\kappa}\E^{\P^{\theta}}[Y-X]=\mathcal{E}^{\kappa}[Y-X]\leq G^\kappa_{0,T}[Y-X].
    \end{align*}
    Then, the coherent risk measure (resp., convex risk measure) $\rho(t,\cdot)$ satisfies our Assumption \ref{assrho}.
\end{remark}

\section*{Acknowledgments}
	This work was supported  by the National Natural Science Foundation of China (No. 12301178), the Natural Science Foundation of Shandong Province for Excellent Young Scientists Fund Program (Overseas) (No. 2023HWYQ-049), the Natural Science Foundation of Shandong Province (No. ZR2023ZD35) and  the Qilu Young Scholars Program of Shandong University. 

 \end{document}